\documentclass{amsart}
\usepackage{amssymb,latexsym}
\theoremstyle{plain}

\newtheorem{proposition}{Proposition}
\newtheorem{lemma}{Lemma}
\theoremstyle{definition}

\newtheorem{example}{Example}

\newtheorem{remark}{Remark}
\newtheorem{question}{Question}

\begin{document}

\title[singular curves]
{Singular curves over a finite field and with many points}
\author{E. Ballico}
\address{Dept. of Mathematics\\
 University of Trento\\
38123 Povo (TN), Italy}
\email{ballico@science.unitn.it}
\thanks{The author was partially supported by MIUR and GNSAGA of INdAM (Italy).}

\begin{abstract}
Recently Fukasawa, Homma and Kim introduced and studied certain projective singular curves
over $\mathbb {F}_q$ with many extremal properties. Here we extend their definition to more general non-rational curves.
\end{abstract}

\maketitle

\section{Introduction}\label{S1}
Fix a prime $p$ and a $p$-power $q$. Recently S. Fukasawa, M. Homma and S. J. Kim introduced a family of singular rational curves defined over $\mathbb {F}_q$, with many singular points
over $\mathbb {F}_q$ and, conjecturally, some extremal properties. In this paper we discuss a similar type of curves, discuss their extremal properties and, in some cases, show that they are, more or less,
the curves introduced in \cite{fhk}. The zeta-function $Z_Y(t)$ of a singular curve $Y$ is explicitly given in terms of the Frobenius on a ``~topological~'' invariant $H^1_c(Y,\mathbb {Q}_{\ell })$ (\cite{d}, \cite{ap}, p. 2).
Hence $Z_Y(t)$ does not detect the finer invariants of the singular points of $Y$ (it does not distinguish between unibranch points defined over the same extension of $\mathbb {F}_q$; in particular it does not distinguish
between a smooth point and a cusp). Using gluing of points of the normalization with the same residue field we may define a ``~minimal~'' singular curve with prescribed normalization and prescribed
zeta-function.

Let $Y$ be a geometrically integral projective curve defined over $\mathbb {F}_q$. Let $u: C\to Y$
denote the normalization. Since any finite field is perfect, $C$ and $u$ are defined over $\mathbb {F}_q$.  Hence for every integer $n\ge 1$ we have $u(C(\mathbb {F}_{q^n}))\subseteq Y(\mathbb {F}_{q^n})$ and for
each $P\in Y(\mathbb {F}_{q^n})$ the scheme $u^{-1}(P)$ is defined over $\mathbb {F}_{q^n}$. Hence the finite set $u^{-1}(P)_{red}$ is defined over $\mathbb {F}_{q^n}$ (but of course if
$\sharp (u^{-1}(P)_{red})>1$ the points of $u^{-1}(P)_{red}$ may only be defined over a larger extension of $\mathbb {F}_q$). We are interested in properties of the set $Y(\mathbb {F}_q)$ knowing $C$.
A. Weil's study of the zeta-function of smooth projective curves was extended to the case of singular curves (\cite{d}). We will use the very useful and self-contained treatment given by Y. Aubry and M. Perret (\cite{ap}).
There are infinitely many curves $Y'$ defined over $\mathbb {F}_q$, with $C$ as their normalization and with the same zeta-function (see Examples \ref{e1}, \ref{e2} and Lemma \ref{e3}). However, given $Y$, there is
one natural such curve if we prescribe also the sets $u^{-1}(P)_{red}$ as subsets of $C(\overline{\mathbb {F}}_q)$. Let $w_s: Y_s \to Y$ be the seminormalization of $Y$ (\cite{d}, \cite{v}).
We recall that $Y_s$ is an integral projective curve with $C$ as its normalization and that $u = w_s\circ u_s$, where $u_s: C \to Y_s$ is the normalization map. Over an algebraically closed
field the one-dimensional seminormal singularities with embedding dimension $n\ge 1$ are exactly the singularities formally isomorphic to the local ring at the origin of  the union of the coordinate axis in $\mathbb {A}^n$.
Even over a finite field the curve we introduce in this note is defined in the same way, i.e. the curves $C_{[q,n]}$, $n\ge 2$, defined below are obtained in the same way, i.e. the gluing process introduced by C. Traverso (\cite{t}) gives
always a seminormal curve and if the base field is algebraically closed, then all seminormal curve singularities are obtained in this way (over an algebraically closed base field a more general construction is given in \cite{s}, p. 70). We call axial singularities the curve singularities obtained in this way. Hence
by definition we say that $(Y,P)$ is an axial singularity with embedding dimension $n$ if and only if over $\overline{\mathbb {F}}_q$ it is formally isomorphic at $P$ to the germ at $0$ of the union of  of the $n$ axis of $\mathbb {A}^n$. An axial singularity of embedding
dimension $n>2$ is not Gorenstein.  An axial singularity of embedding
dimension $2$ is an ordinary double point except that over a non-algebraically closed base field, say $\mathbb {F}_q$, we need to distinguish if the two branches of $Y$ at $P$ (or the two lines of its tangent cone)
are defined over $\mathbb {F}_q$ or not (in the latter case each of them is defined over $\mathbb {F}_{q^2}$). Similarly, for an axial singularity $(Y,P)$ of embedding dimension $t\ge 2$ defined over $\mathbb {F}_q$, the $t$ lines of the tangent cone $C(P,Y)$
are defined over $\mathbb {F}_{q^t}$ and their union is defined over $\mathbb {F}_q$. In the examples we are interested in, none of these lines will be defined over a field $\mathbb {F}_{q^e}$ with $e<t$. If $P\in Y$ is a singular point, then we may associate a non-negative
integer $p_a(Y,P)$ (usually called the arithmetic genus of the singularity or the drop of genus the singular point $P$) such that $p_a(Y) = p_a(C) + \sum _{P\in \mbox{Sing}(Y)} p_a(Y,P)$. 
When $Y$ is an axial singularity with embedding dimension $n$, then $p_a(Y,P) =n-1$.

Let $C$ be a smooth and geometrically connected projective curve defined over $\mathbb {F}_q$. Let $F_q: C(\overline{\mathbb {F}}_q) \to C(\overline{\mathbb {F}}_q)$
be the action of the Frobenius of order $q$.  For each $P\in C(\overline{\mathbb {F}}_q)$ let $ord(P,q)$ be the cardinality of the orbit of $P$ by the action
of $F_q$. For every integer $t\ge 1$ we have we have $F_{q^t} =(F_q)^t$ and $C(\mathbb {F}_{q^t}) =\{P\in C(\overline{\mathbb {F}}_q): (F_q)^t(P) =P\}$.
Hence $ord(P,q)$ is the minimal integer $t$ such that $P\in C(\mathbb {F}_{q^t})$ and $P\in C(\mathbb {F}_{q^s})$ if and only if $ord(P,q)|s$.

We fix $q$, $C$ and an integer $n\ge 2$. For all integers $i\ge 1$ set $N_i:= \sharp (C(\mathbb {F}_{q^i}))$. Let $N'_i$ be the number
of all $P\in C(\overline{\mathbb {F}}_q)$ with $ord (P,q) =i$. Since $N_t =\sum _{s|t} N'_s$, M\"{o}bius inversion formula gives $N'_t = \sum _{s|t} \mu (s)N_{t/s}$ for all $t$.

We construct a singular curve $C_{[q,n]}$ with $C$ as its normalization and $\sharp (C_{[q,n]}(\mathbb {F}_q))$ very large in the following way.  Fix an integer $t$ such that $2\le t\le n$.
For each $P\in C(\mathbb {F}_{q^t})$ with $ord (P,q)=t$ the orbit of the Frobenius $F_q$ has order $t$, say $\{P,\dots ,F_q^{t-1}(P)\}$. Let $C_{[q,n]}$ be the only curve obtained by gluing
each of these orbits (for all possible $t\le n$) (see Remark \ref{t1}). By construction $C_{[q,n]}$ is a seminormal curve defined over $\mathbb {F}_q$, each singular points of $C_{[q,n]}$ is defined over
$\mathbb {F}_q$ and $\sharp (C_{[q,n]}(\mathbb {F}_q)) = N_1 + \sum _{i=2}^{n} {N'_i/i}$. The integers $N'_i$, $i\ge 1$, are uniquely determined
by the integers $N_i$, $i\ge 1$, because $N_t = \sum _{s|t} N'_t$and hence $\sum _{s|t} \mu (s)N_{t/s}$. Fix $P\in C_{[q,n]}$ with embedding dimension $t\ge 2$. The
Frobenius $F_q$ acts on the local ring $\mathcal {O}_{C_{[q,n]},P}$ and hence on the
$t$ branches of $C_{[q,n]}$ at $P$ (i.e. the $t$ smooth branches through $0$ of the tangent cone of $C_{[q,n]}$ at $P$). Since $P$ is an axial singularity, the action of Frobenius is the restriction to $u^{-1}(P)$  of the action of the Frobenius $F_q: C(\overline{\mathbb {F}}_q) \to C(\overline{\mathbb {F}}_q)$.
Hence this action is cyclic, i.e. it has a unique orbit. Thus if $O = u(P)$ with $ord (P) =t$, then $p_a(C_{[q,n]})=t-1$ and
none of the $t$ branches of $C_{[q,n]}$ at $u(O)$ is defined over a proper subfield of $\mathbb {F}_{q^t}$. See Propositions \ref{b1}, \ref{b2}, Question \ref{b3} and Remark \ref{b4} for the relations between $\mathbb {P}^1_{[q,n]}$ and the curves $B$ and $B_n$ studied in \cite{fhk}.

\section{The curves $C_{[q,n]}$ and their maximality properties}\label{S2}
Let $u: C \to Y$ denote the normalization map. We often write $u^{-1}(P)$ instead of $u^{-1}(P)_{red}$.
Set $\Delta _Y:=
\sharp (u^{-1}(\mbox{Sing}(Y) (\overline{\mathbb {F}}_q))) - \sharp (\mbox{Sing}(Y)(\overline{\mathbb {F}}_q))$. The
zeta-function
$Z_Y(t)$ of $Y$ is the product of the zeta-function $Z_C(t)$ of $C$ and a degree $\Delta _Y$ polynomials whose inverse roots
are roots of unity (\cite{d}, \cite{ap}, Theorem 2.1 and Corollary 2.4). Let $\omega _i$, $1 \le i \le 2g$, be the inverse
roots of numerator of $Z_C(t)$ and $\beta _j$, $1\le i \le \Delta _Y$ the inverse roots
of the polynomial $Z_Y(t)/Z_C(t)$. For every integer $n \ge 1$ we have
\begin{equation}\label{eqb1}
\sharp (Y(\mathbb {F}_{q^n})) = q^n+1 - \sum _{i=1}^{2g}\omega _i^n
- \sum _{j=1}^{\Delta _Y} \beta _j^n
\end{equation}
We have $\sharp (C(\mathbb {F}_{q^n})) = q^n+1 - \sum _{i=1}^{2g}\omega _i^n$. Recall that $\vert \beta _j\vert = 1$ for all
$j$. Assume for the moment that $n$ is odd. In this case among all curves with fixed normalization $C$ and with fixed $\Delta _Y$ the integer
$\sharp (Y(\mathbb {F}_q))$ is maximal (resp. minimal) for a curve with
$\beta _j=-1$ for all $j$ (resp. $\beta _j=1$) for all $j$), if any such curve exists. In $n$ is even them the minimum is achieved if there is $Y$ with $\beta _j\in \{-1,1\}$ for all $j$.

\begin{lemma}\label{e0}
Let $Y$ be a geometrically integral projective curve and $u: C\to Y$ its normalization. The degree $\Delta _Y$ polynomial $Z_Y(t)/Z_C(t)$ has all its inverse roots equal to $-1$ if and only
if for each $P\in \mbox{Sing}(Y)$ either $\sharp (u^{-1}(P)) =1$ or $P\in Y(\mathbb {F}_q)$ and $u^{-1}(P)$ is formed by two points of $C(\mathbb {F}_{q^2})$ (in the latter case
these two points are exchanged by the Frobenius and they are in $C(\mathbb {F}_{q^2})\setminus C(\mathbb {F}_q)$).
\end{lemma}

\begin{proof}
The explicit form of the polynomial $Z_Y(t)/Z_C(t)$ is given in \cite{ap}, Theorem 2.1. The polynomial $Z_Y(t)/Z_C(t)$ is a product of polynomials, each of them associated
to a different singular point of $Y$. Hence it is sufficient to consider separately the contribution of each singular point of $Y$. Fix $P\in \mbox{Sing}(Y)$
and call $Z_P(t)$ the associated polynomial. Let $d_P$ be the minimal integer $t\ge 1$ such that $P\in Y(\mathbb {F}_{q^t})$. We have $(1-t^{d_P})Z_P(t)
= \prod _{Q\in u^{-1}(P)}(1-t^{ord(Q,q)})$. Since $Y$ is defined over $\mathbb {F}_q$, the orbit of
$P$ by the Frobenius of $Y$ has order $d_P$. For any point $P' \ne P$ in this orbit, say $F_q^x(P)$ for some $x\in \{1,\dots ,d_P-1\}$ we have
$u^{-1}(P') = F_q^x(u^{-1}(P))$ and $d_{P'} =d_P$. Since the normalization map is defined over $\mathbb {F}_q$, we have $d_P|ord (Q,q)$ for each $Q\in u^{-1}(P)$. 

First assume $\sharp (u^{-1}(P))=1$.
The only point, $Q$, of $u^{-1}(P)$ is defined over $\mathbb {F}_{q^{d_P}}$. Since $d_{u(Q)}| ord (Q,q)$, we get $ord (Q,q)=d_P$. We easily get
that $\sharp (u^{-1}(P))=1$ if and only if the constant $1$ is the factor of $Z_Y(t)/Z_C(t)$ associated to the orbit of $P$.
Hence from now on we assume $\alpha :=\sharp (u^{-1}(P)) \ge 2$. 

If $ord (Q,q)>d_P$ for some $Q\in u^{-1}(P)$ and either $d_P \ge 2$ or $ord (Q,q)\ge 3$, then we
get that $(1-t^{d_P})Z_P(t)$ has a root of order $>\max \{d_P,2\}$ and hence $Z_P(t)$ has a root $\ne -1$.

Now assume $d_P\ge 2$ and $ord (Q,q) = d_P$ for all $Q\in u^{-1}(Q)$. We get $Z_P(t) = (1-t^{d_P})^\alpha$. Since
we assumed $\alpha \ge 2$, even in this case $Z_P(t)$ has a root $\ne -1$. 

Now assume $d_P=1$. It remains
to analyze the case $ord (Q,q)\in \{1,2\}$ for any $Q\in u^{-1}(P)$. If $ord (Q,q)=1$ for at least one $Q\in u^{-1}(P)$, then $Z_P(1)=0$. If $ord (Q,q)=2$ for all $Q\in u^{-1}(P)$, then
$Z_P = (1+t)(1-t^2)^{\alpha -1}$ has $\alpha -1$ roots equal to $1$.\end{proof}

\begin{remark}\label{b2.0}
Fix $q$, $g$, $C$ with genus $g$ and an integer $n\ge 2$. Set $N_i:= \sharp (C(\mathbb {F}_{q^i})$. We have $\sharp (C_{[q,n]}(\mathbb {F}_q))
= N_1+\sum _{i=2}^{n} N'_i/i$. 
Now assume that $g>0$ and that $q$ is a square. If $C$ is a minimal curve for $\mathbb {F}_q$, then it is a minimal curve for each
$\mathbb {F}_{q^i}$, $i\ge 2$ (use that $C$ is minimal if and only if $Z_C(t) = \frac{(t-q)^{2g}}{(1-t)(1-qt)}$ (\cite{vr})). Hence for fixed $q$ and $n$ with $n$ an odd prime
the integer $\sharp (C_{[q,n]}(\mathbb {F}_q)) $
is minimal varying $C$  among all smooth curves of genus $g$ if and only if $C$ is a minimal curve. If $n=2$ and $q$ is a square, then $Z_Y(t)$ is the same for
all minimal curves over the same genus. Since $N'_2 = N_2-N_1$, we get $N_1+N'_2/2 = N_2+N_1/2$ and hence when $n=2$ and $q$ is a square for
a fixed genus $g$ the minimal among all $\sharp (C_{[q,2]}(\mathbb {F}_q)) $ with fixed genus $g$ is obtained if and only if $C$ is a minimal curve.
\end{remark}

\begin{proposition}\label{b1}
The curve $\mathbb {P}^1_{[q,2]}$ is isomorphic over $\mathbb {F}_q$ to the plane curve $B \subset \mathbb {P}^2$ defined in \cite{f} and \cite{fhk}.
\end{proposition}

\begin{proof}
Let $u: \mathbb {P}^1 \to \mathbb {P}^1_{[q,2]}$ denote the normalization map. The normalization map $\Phi: \mathbb {P}^1\to B$ is unramified, because
the composition of it with the inclusion $B\hookrightarrow \mathbb {P}^2$ is unramified (part (i) of \cite{fhk}, Theorem 2.2). By \cite{fhk}, part (iii) of Theorem 2.2, $B$ is a degree $q+1$ plane curve with $(q^2-q)/2$ singular points
and $\Phi (P) = \Phi (Q)$ with $P\ne Q$ if and only if $u(P)=u(Q)$. Hence the universal property of the seminormalization gives the existence of a morphism $\psi : \mathbb {P}^1_{[q,2]}
\to B$ such that $\psi$ is a bijection. Since $p_a( \mathbb {P}^1_{[q,2]}) = (q^2-q)/2$, we have $p_a( \mathbb {P}^1_{[q,2]}) = p_a(B)$. Hence $\psi$ is an isomorphism.
\end{proof}

\begin{proposition}\label{b2}
Fix an integer $n\ge 3$. Then $\mathbb {P}^1_{[q,n]}$ is the seminormalization of the curve $B_n\subset \mathbb {P}^n$, defined in \cite{fhk}, \S 6, and
there is a birational morphism $\psi _{q,n}: \mathbb {P}^1_{[q,n]} \to B_n$ defined over $\mathbb {F}_q$ such that $\psi _{n,q}: \mathbb {P}^1_{[q,n]}(K) \to B_n(K)$
is bijective for every field $K\supseteq \mathbb {F}_q$.
\end{proposition}

\begin{proof}
Let $u: \mathbb {P}^1 \to \mathbb {P}^1_{[q,n]}$ and $\Phi _n: \mathbb {P}^1\to B_n$ denote the normalization maps. By \cite{fhk}, part (ii) of Theorem 6.4, each point $P\in \mbox{Sing}(B_n)$
corresponds to an integer $t\in \{2,\dots ,n\}$ and an orbit of the Frobenius on $\mathbb {P}^1(\mathbb {F}_{q^t})\setminus \mathbb {P}^1(\mathbb {F}_{q^{t-1}})$.
Hence the definition of $\mathbb {P}^1_{[q,n]}$ and the universal property of the seminormalization gives a birational morphism $\psi _{q,n}: \mathbb {P}^1_{[q,n]} \to B_n$ defined over $\mathbb {F}_q$
such that $\psi _{n,q}: \mathbb {P}^1_{[q,n]}(K) \to B_n(K)$
is bijective for every field $K\supseteq \mathbb {F}_q$.
\end{proof}

\begin{question}\label{b3}
We guess that $\psi _{q,n}$ is an isomorphism.
\end{question}

\begin{remark}\label{b4}
Fix a prime power $q$ and the integer $n\ge 3$. Let $\Phi _n: \mathbb {P}^1\to B_n$ denote the normalization map. By
\cite{fhk}, part (i) of Theorem 6.4, $\Phi _n$ is unramified (this is a necessary condition for being
$\psi _{q,n}$ an isomorphism). The following conditions are equivalent:
\begin{itemize}
\item[(i)] the morphism $\psi _{q,n}$ is an isomorphism;
\item[(ii)] $p_a(B_n) = p_a(\mathbb {P}^1_{[q,n]})$;
\item[(iii)] for each $P\in \mbox{Sing}(B_n)$, say with $P =\Phi _n(Q)$ and $ord(Q,q) =t$, the singularity $(B_n,P)$
has arithmetic genus $t-1$;
\item[(iv)]  for each $P\in \mbox{Sing}(B_n)$, say with $P =\Phi _n(Q)$ and $ord (Q,q) =t$ the tangent cone $C(P,B_n)\subset \mathbb {P}^n$
is formed by $t$ lines through $P$ spanning a $t$-dimensional linear subspace.
\end{itemize}
Part (iv) is just the definition of seminormal singularity given in \cite{d}.
Since $\Phi _n$ is unramified, $B_n$ has at $P$ $t$ smooth branches.\end{remark}

\begin{proposition}\label{a1}
Let $C$ be a smooth and geometrically irreducible projective curve defined over $\mathbb {F}_q$. Set $2\delta :=  \sharp (C(\mathbb {F}_{q^2}))
-\sharp (C(\mathbb {F}_q))$. Let $Y$ a projective curve defined over $\mathbb {F}_q$ with $C$ as its normalization. We have $\sharp (Y(\mathbb {F}_q))
\ge \sharp (C(\mathbb {F}_q)) +\delta$ and $p_a(Y) \le g+\delta$ if and only if $Y$ is isomorphic to $C_{[q,2]}$ over $\mathbb {F}_q$.
\end{proposition}

\begin{proof}
The ``~if~'' part is true, because $p_a(C_{[q,2]}) =g+\delta$ and $\sharp (C_{[q,2]}(\mathbb {F}_q)) = g+\delta$. Assume $\sharp (Y(\mathbb {F}_q))
\ge \sharp (C(\mathbb {F}_q)) +\delta$ and $p_a(Y) \le g+\delta$. Let $u: C \to Y$ be the normalization map. The morphism $u$ is defined over $\mathbb {F}_q$, i.e. over
a field on which $Y$ is defined, because any finite field is perfect. We have
$\sharp (\mbox{Sing}(Y)) \le p_a(Y)-\delta$ and equality holds only if each singular point of $Y$ is formally isomorphic over $\overline{\mathbb {F}}_q$ either to a node or an ordinary
cusp. Set $\Delta _Y:= \sharp (u^{-1}(\mbox{Sing}(Y)(\overline{\mathbb {F}}_q)
-\sharp (\mbox{Sing}(Y)(\overline{\mathbb {F}}_q)$. The polynomial $Z_Y(t)/Z_C(t)$ has degree $\Delta _Y$ and
$\sharp (Y(\mathbb {F}_q)) \le \sharp (C(\mathbb {F}_q)) +\Delta _Y$ and equality holds if and only if each inverse root of $Z_Y(t)/Z_C(t)$ is equal to $-1$.
Since $\Delta _Y \le p_a(Y) -g$, we get $\Delta _Y = \delta$ and $p_a(Y) = g+\delta$. Since $p_a(Y) =g+\Delta _Y$ and $\sharp (\mbox{Sing}(Y)(\mathbb {F}_q)) \ge \delta$,
we get $\mbox{Sing}(Y)(\overline{\mathbb {F}}_q) = \mbox{Sing}(Y)(\mathbb {F}_q)$, $\sharp (\mbox{Sing}(Y)(\mathbb {F}_q))=\delta$ and that for each $P\in \mbox{Sing}(Y)(\mathbb {F}_q)$ the set $u^{-1}(P)$ is formed by two points of $C(\mathbb {F}_{q^2})\setminus C(\mathbb {F}_q)$
exchanged by the Frobenius (Lemma \ref{e0}). Since $p_a(Y)=g+\sharp (\mbox{Sing}(Y)(\overline{\mathbb {F}}_q))$ and $\sharp (u^{-1}(P))\ge 2$
for each $P\in \mbox{Sing}(Y)(\overline{\mathbb {F}}_q)$, we also get that $Y$ is nodal. Hence $Y$ is seminormal. The structure
of the fibers of $u^{-1}(P)$, $P\in \mbox{Sing}(Y)(\overline{\mathbb {F}}_q)$, gives $Y = C_{[q,2]}$.\end{proof}

\begin{proposition}\label{a2}
Let $Y$ be a geometrical integral projective curve defined over $\mathbb {F}_q$ and with only seminormal singularity. Let $u: C \to Y$ be the normalization map. Let $\Delta _Y$ be the degree
of the polynomial $Z_Y(t)/Z_C(t))$. Assume $2\Delta _Y \le  \sharp (C(\mathbb {F}_{q^2}))
-\sharp (C(\mathbb {F}_q))$. We have $\sharp (Y(\mathbb {F}_q)) \le \sharp (C(\mathbb {F}_{[q,2]}))$ and equality holds if and only if $Y$ is isomorphic to $C_{[q,2]}$ over $\mathbb {F}_q$.
\end{proposition}

\begin{proof}
We have $\sharp (Y(\mathbb {F}_q)) \le \sharp (C(\mathbb {F}_q)) +\Delta _Y$ and equality holds if and only if each inverse root of $Z_Y(t)/Z_C(t)$ is equal to $-1$.
Hence we may assume $2 \Delta _Y =  \sharp (C(\mathbb {F}_{q^2}))
-\sharp (C(\mathbb {F}_q))$. Since $Y$ has only seminormal singularities, $u$ is unramified. Since $u$ is unramified, we have $\sharp (u^{-1}(P)) \ge 2$ for all $P\in \mbox{Sing}(Y)$. Hence Lemma \ref{e0} gives
that the fibers of $u$ are the fibers of the normalization map $C \to C_{[q,2]}$. Since $Y$ and $C_{[q,2]}$ are seminormal, we get that they are isomorphic. They are isomorphic over $\mathbb {F}_q$, because
$u$ is defined over $\mathbb {F}_q$ and the seminormalization is defined over $\mathbb {F}_q$.
\end{proof}

\begin{remark}\label{a3}
In the case $C\cong \mathbb {P}^1$ Propositions \ref{a1} and \ref{a2} are partial answers to a questions raised in \cite{fhk}, Remark 2.5. Examples \ref{e1}, \ref{e2} and Lemma \ref{e3} show that we need to add some conditions on the curve
$Y$, not only to fix the normalization $\mathbb {P}^1$ and assume $\Delta _Y \le (q^2-q)/2$.
\end{remark}

\begin{example}\label{e1}
Fix a geometrically integral projective curve  $A$ defined over $\mathbb {F}_q$ and $P\in A(\mathbb {F}_q)$. Now we define a geometrically integral curve $Y$ defined over $\mathbb {F}_q$
and a morphism $v: A \to Y$ defined over $\mathbb {F}_q$, such that $v\setminus A\setminus \{P\}$ is an isomorphism onto $Y\setminus v(P)$, but $v$ is not a isomorphism. Notice
that for each such pair $(Y,v)$ we would have $p_a(Y)>p_a(A)$ and that for every integer $t\ge 1$ $v$ induces a bijection $A(\mathbb {F}_{q^t}) \to Y(\mathbb {F}_{q^t})$.
To define $Y$ and $v$ it is sufficient to define them in a neighborhood of $P$ in $A$ and the glue to it the identity map $A\setminus \{P\} \to A\setminus \{P\}$. Fix an embedding
$j: A\hookrightarrow \mathbb {P}^r$, $r\ge 3$, and take a projection of $j(A)$ into $\mathbb {P}^2$ from an $(r-3)$-dimensional linear subspace not containing $j(P)$, but
intersecting the Zariski tangent space of $j(A)$ at $j(P)$.
\end{example}

\begin{example}\label{e2}
Fix a geometrically integral projective curve $A$ defined over $\mathbb {F}_q$ and any point $P\in A_{reg}(\overline{\mathbb {F}}_q)$. Let $t$ be the minimal integer $t\ge 1$
such that $P\in A(\mathbb {F}_{q^t})$, i.e. let $t$ be the cardinality of the orbit of $P$ by the action of the Frobenius. We assume $t\ge 2$, because the case $t=1$ is covered by Example \ref{e1}. Hence the orbit of $P$ by the action of the Frobenius $F_q$ has order $t$ (it is $\{P,F_q(P),\dots ,F_q^{t-1}(P)\}$). Let $Y$ denote the only curve and $v: A \to Y$ the only morphism obtained in the following way.
We fix a bijection of sets $v: A \to Y$ and use it to define a topology on the set $Y$. Now we define $Y$ as a ringed space. On $Y\setminus u(\{P,F_q(P),\dots ,F_q^{t-1}(P)\})$ we assume that $v$ is an isomorphism
of local ringed spaces. For each $Q\in \{P,F_q(P),\dots ,F_q^{t-1}(P)\}$ we impose that $\mathcal {O}_{Y,Q}$ is the local ring of a unibranch singular point and that $v$ is the normalization map
(it may be done using the method of Example \ref{e1} with $Q$ instead of $P$ and $\mathbb {F}_{q^t}$ instead of $\mathbb {F}_q$).
We need to do the construction simultaneously over all $Q\in \{P,F_q(P),\dots ,F_q^{t-1}(P)\}$ and in such a way that the morphism is defined over
$\mathbb {F}_q$. As in Example \ref{e1} it is sufficient to define $v\vert U$, where $U$ is a neighborhood of $\{P,F_q(P),\dots ,F_q^{t-1}(P)\}$. There is an embedding
$j: A \to \mathbb {P}^r$, $r\ge t+2$, such that the $t$ lines $T_Q(j(A)$, $Q\in \{P,F_q(P),\dots ,F_q^{t-1}(P)\}$, are linearly independent. Since $j$ is defined over $\mathbb {F}_q$ the Frobenius $F_q$ of $\mathbb {P}^r$
acts  on $j(A)$ and on the tangent developable of $A$. Since $j(P)$ is defined over $\mathbb {F}_{q^t}$. Hence $T_{j(P)}j(A)(\mathbb {F}_{q^t})\setminus j(P)$ has $(q^{t+1}-1)/(q-1) -1$ elements. Fix any
$O\in T_{j(P)}j(A)(\mathbb {F}_{q^t})\setminus j(P)$. For each $x\in \{1,\dots ,t-1\}$, $F_q^x(O) \in T_{j(F_q^x(P))}j(A)(\mathbb {F}_{q^t})\setminus j(F_q^x(P))$. Since the $t$ tangent lines are linearly independent,
the linear space $E:= \langle \{O,F_q(O),\dots ,F_q^{t-1}(O)\}\rangle$ has dimension $t-1$. Since $E$ is $F_q$-invariant, it is defined over $\mathbb {F}_q$. Let $\pi : \mathbb {P}^r\setminus E\to \mathbb {P}^{r-t}$ denote the linear
projection from $E$. Since $E$ is defined over $\mathbb {F}_q$, $\pi$ is defined over $\mathbb {F}_q$. Hence the integral projective curve $T:= \overline{\pi (j(A)\setminus E\cap j(A))}\subset \mathbb {P}^{r-t}$ is defined
over $\mathbb {F}_q$. Since the $t$ tangent lines are linearly independent
and $O\ne j(P)$, we have $E\cap \{j(P),\dots ,j(F_q^{t-1}(P))\} = \emptyset$. Hence $E \cap j(U)=\emptyset$ for a sufficiently small neighborhood $U$ of $\{P,F_q(P),\dots ,F_q^{t-1}(P)\}$. Assume for the moment
that $\pi \vert j(A)\setminus j(A)\cap E$ is birational onto its image. Since $\pi \vert j(A)\setminus j(A)\cap E$ is birational onto its image, it is separable. Hence only finitely many points of $j(A_{reg})$ have a tangent line intersecting $E$. Restricting if necessary
$U \subseteq A_{reg}$ we may assume that for no other point $Q\in j(U)(\overline{\mathbb {F}}_q)$ the Zariski tangent space $T_{j(Q)}(j(A))$ intersects $E$. Since  $\pi \vert j(A)\setminus j(A)\cap E$ is birational onto its image, it is generically
injective. Hence restricting $U\subseteq A_{reg}$ we may assume that $\pi \vert j(U)$ is injective and an isomorphism outside $\{j(P),j(F_q(P)),\dots ,j(F_q^{t-1}(P))\}$. At these points the curve $T$ has a cusp, but perhaps not an ordinary cusp,
i.e. it is a unibranch singular point. Hence to conclude the example it is sufficient to find $j$ such that $\pi \vert j(A)\setminus j(A)\cap E$ is birational onto its image.
We take as $j$ is a linearly normal embedding of degree $d >  \max \{2p_a(A)-2,p_a(A)+t\}$. Since $d >  \max \{2p_a(A)-2,p_a(A)+t\}$, Riemann-Roch gives $r = d-p_a(A)$. Assume that $\pi \vert j(A)\setminus j(A)\cap E$ is not birational onto its image and call $x\ge 2$ its degree.
Thus $\deg (T) \le d/x \le d/2$. Since $j(A)$ spans $\mathbb {P}^r$, $T$ spans $\mathbb {P}^{r-t}$. Hence $\deg (T) \ge r-t = d-p_a(A)-t$. Hence $(d-p_a(A)-t)\ge 2(d-p_a(A) -t$, contradicting our assumption $d> p_a(A)+t$.
\end{example}

\begin{lemma}\label{e3}
Fix an integer $y>0$. Let $A$ be a geometrically integral projective curve defined over $\mathbb {F}_q$. Assume $A_{reg}(\mathbb {F}_q)\ne \emptyset$ and fix $P\in A_{reg}(\mathbb {F}_q)$.
Then there are a geometrically integral projective curve $Y$ and a morphism $u: A \to Y$ defined over $\mathbb {F}_q$ such that $u$ induces an isomorphism of $A\setminus \{P\}$
onto $Y\setminus u(P)$ and $p_a(Y)=p_a(A)+y$.
\end{lemma}

\begin{proof}
Let $\bf {m}$ be the maximal ideal of the local ring $\mathcal {O}_{A,P}$. By assumption  $\mathcal {O}_{A,P}/{\bf {m}} \cong \mathbb {F}_q$ and
$\mathbb {F}_q\cdot 1\subset \mathcal {O}_{A,P}$. Hence the $\mathbb {F}_q$-vector space $\mathcal {O}_{A,P}$ is the direct sum of its subspaces
$\mathbb {F}_q\cdot 1$ and $\bf {m}$. Set $\mathcal {O}_{Y,u(P)}:= \mathbb {F}_q\cdot 1 +\bf {m}^{y+1} \subset \mathcal {O}_{A,P}$. It is easy to check that $\mathcal {O}_{Y,u(P)}$ is a local ring with
${\bf {m}}^{y+1}$ as its maximal ideal. Since $P\in A_{reg}$, $\mathcal {O}_{A,P}$ is a DVR. Hence $\bf {m}/\bf {m}^{t+1}$ is a $\mathbb {F}_q$-vector space of dimension $y$. We take as $Y$ the same topological space
as $Y$, but with $\mathcal {O}_{Y,u(P)}$ at the point $u(P)$ associated to $P$ instead of $\mathcal {O}_{A,P}$. With this definition of $u$ we have $\dim _{\mathbb {F}_q}(u_\ast (\mathcal {O}_A)/\mathcal {O}_Y)=y$.
Hence $p_a(Y)=p_a(A)+y$.
\end{proof}

\begin{remark}\label{t1}
Fix $q$, $C$ and an integer $n \ge 2$. Here we explain one way to check the existence of the curve $C_{[q,n]}$. We obtain $C_{[q,n]}$ in finitely many steps each of them similar to the one described
in Example \ref{e2}. We use $z$ steps, where $z$ is the number of orbits of $F_q$ in $C(\mathbb {F}_{q^n})\setminus C(\mathbb {F}_q)$. At each of the steps we glue together one of these orbits.
We do not need any notion of gluing, except that set-theoretically in each step one of these orbits is sent to a single point and for all other points the map is an isomorphism.
Fix $Q\in C(\mathbb {F}_{q^n})\setminus C(\mathbb {F}_q)$ and assume $ord (Q,q) =t|n$. Hence $\{Q,F_q(Q),\dots ,F_q^{t-1}(Q)\}$
is the orbit of $Q$ for the action of $F_q$. Call $A$ the geometrically integral curve arising in the steps at which we want to glue this orbit. Hence there is a geometrically integral projective $A$
curve defined over $\mathbb {F}_q$ with $C$ as its normalization (call $u: C \to A$) and $u(Q)\in A_{reg}$ (in the previous steps (if any) the maps where isomorphism at
each point of $\{Q,F_q(Q),\dots ,F_q^{t-1}(Q)\}$). Set $P:= u(Q)$. Since $u$ is defined over $\mathbb {F}_q$ and $u$ is an isomorphism in a neighborhood
of $u^{-1}(\{P,F_q(P),\dots ,F_q^{t-1}(P)\})$, we have $\{P,F_q(P),\dots ,F_q^{t-1}(P)\} \subset A_{reg}$ and these $t$ points are distinct. Hence $P\in A_{reg}(\mathbb {F}_{q^t}$
and $P\in  A_{reg}(\mathbb {F}_{q^y}$ if and only if $t|y$.
As in Example \ref{e2} we get several curves $Y$ and morphism $v: A\to Y$ defined over $\mathbb {F}_q$, sending $\{P,F_q(P),\dots ,F_q^{t-1}(P)\}$ to a single point, $O$,
of $Y$ and induces an isomorphism of $A\setminus \{P,F_q(P),\dots ,F_q^{t-1}(P)\}$ onto $Y\setminus \{O\}$. Let $A_1$ be the seminormalization of $Y$ in $A$. Then we use $A_1$ instead of $A$.
After $z$ steps we get $C_{[q,n]}$. To get $C_{[q,n]}$ we use the existence of the seminormalization. The result does not depend from the order of the gluing. Hence $C_{[q,n]}$ depends
only from $q$, $C$ and $n$. Hence the curves $\mathbb {P}^1_{[q,n]}$ depends only from $q$ and $n$.
\end{remark}

\providecommand{\bysame}{\leavevmode\hbox to3em{\hrulefill}\thinspace}

\end{document}